\documentclass[12pt, twoside, leqno]{article}
 \pdfoutput=1
\usepackage{amsmath,amsthm}
\usepackage{amssymb,latexsym}
\usepackage{enumerate}
\usepackage{lscape}

\newtheorem{teo}{Theorem}[section]
\newtheorem{cor}[teo]{Corollary}
\newtheorem{lemmino}[teo]{Lemma}
\newtheorem{prop}[teo]{Proposition}

\theoremstyle{definition}
\newtheorem{defin}[teo]{Definition}
\newtheorem{oss}[teo]{Remark}

\numberwithin{equation}{section}

\begin{document}


\baselineskip=17pt


\title{An upper bound for the genus
 of a curve  without points of small degree}

\author{Claudio Stirpe\\
Dipartimento di Matematica, Universit\`a di Roma 'Sapienza'\\
Piazzale Aldo Moro 00185 Rome, Italy\\
E-mail: stirpe@mat.uniroma1.it}

\date{February 2012}

\maketitle


\renewcommand{\thefootnote}{}



\renewcommand{\thefootnote}{\arabic{footnote}}
\setcounter{footnote}{0}


\begin{abstract}
In this paper we prove that for any prime $p$ there is a constant $C_p>0$
such that for any $n>0$ and for any $p$-power $q$ there is a smooth, projective,
absolutely irreducible curve over $\mathbb{F}_q$ of genus $g\leq C_p q^n$ without 
points of degree smaller than $n$.
\end{abstract}

\section{Introduction}

Let $X$ be a smooth, projective, absolutely irreducible
curve over the finite field $\mathbb{F}_q$ and let $K$
be the function field of $X$.
For any integer $n>0$ let $a_n$ denote the number of places
of $K$ of degree~$n$. 
Then $N_n= \sum_{d|n} d a_d$ is the number of rational points over the constant field
extension $K \mathbb{F}_{q^n}$. The Weil inequality (see \cite{we71}) states that
$$|N_n-q^n-1|\leq 2 g q^{n/2},$$
where $g$ is the genus of the curve. 
A search for curves with many points, motivated by applications
in coding theory, showed that this bound is optimal when the genus $g$
is small compared to $q$ (see \cite{futo96} for further details).
When $g$ is large compared to $q$ sharper estimates hold (see for example \cite{ih81} 
for an asymptotic result or also \cite{st93} Chapter V).
A similar problem arises from finding curves without points of degree $n$ when
$n$ is a positive integer. In particular when $X$ has
no points over $\mathbb{F}_{q^n}$ then $g\geq \frac{q^n-1}{2 q^{n/2}}.$ 
The genus $2$ case was already considered in \cite{mana}. Moreover in a recent paper,  E. Howe,
K. Lauter and J. Top \cite{hola05} show that the previous bound
is not always sharp when $n=1$ and $g=3$ or $4$.
In the same paper they cite an unpublished result of P. Clark and N. Elkies
that states that for every fixed prime $p$ there is a constant
$C_p>0$ such that for any integer $n>0$,
there is a projective curve over $\mathbb{F}_p$ of genus
$g\leq C_p n p^n$ without places of degree smaller than $n$.

In this paper we prove that this bound is not optimal. 
In fact we prove the following result.

\begin{teo}\label{gol}
For any prime $p$ there is a constant $C_p>0$ such that 
for any $n>0$ and for any power $q$ of $p$ there is a projective curve over $\mathbb{F}_q$ of genus
$g\leq C_p q^n$ without points of degree strictly smaller than $n$.
\end{teo}

We show the existence of such curves by means of class field theory. 
The basic facts and definitions about this topic are showed in the next section. 
 In the third section we generalize a result in \cite{arta} about the number of ray class field extensions with
given conductor $\mathfrak{m}$ and we prove some consequences concerning cyclic extensions. In section 4 the
estimate of Theorem \ref{gol} is proved. A table of examples for $q=2$ and $n<20$ is given at the end of the paper.

\section{Background and notation}

Througout the paper we consider the function fields associated to the projective, non singular, geometrically
irreducible curves over the finite field $\mathbb{F}_q$ of characteristic $p$.

The set of the places of the function field $K$ is denoted by $\mathcal{P}_K$ and the set of
divisors of $K$ is denoted by $\mathcal{D}_K$. The degree zero divisors
are denoted by $\mathcal{D}^0_K$. We can associate to every element $z\in K$
 its principal divisor $(z)\in \mathcal{D}^0_K$.
The set of principal divisors is denoted by $Prin(K)$.
The number $h_K=|\mathcal{D}^0_K/Prin(K)|$ is finite
and it is called the divisor class number of $K$. 

We denote by $J_K$ and $C_K$ the idele group and the class group of $K$ (see \cite{nx}, Chapter 2).

In the sequel we use ray class fields for constructing curves. Let $S$
be a finite non-empty set of places and let 
$\mathfrak{m}=\sum n_P P$ be an effective divisor of the
function field $K$ with support disjoint
from $S$. The $S$-congruence subgroup modulo $\mathfrak{m}$ is the subgroup of $J_K$
$$J^\mathfrak{m}_S=\prod_{P\in S} \hat{K}^*_P \times \prod_{P\not\in S} \hat{U}_{P}^{(n_P)}.$$
where $\hat{K}^*_P$ is the completion of $K$ at the place $P$ and 
$\hat{U}_P^{(n_P)}$ is the $n_P$-th unit group.

\begin{defin}\label{rclgroup}
A ray class group is a subgroup  $C^\mathfrak{m}_S$ of $C_K$ of the form 
$$C^\mathfrak{m}_S=(K^* J^\mathfrak{m}_S)/K^*$$
where $J_S^\mathfrak{m}$ is a $S$-congruence subgroup modulo $\mathfrak{m}$.
\end{defin}

The index of $C^\mathfrak{m}_S$ in $C_K$ is finite 
and we denote by $K^\mathfrak{m}_S$ the function field associated
to the subgroup $C^\mathfrak{m}_S$ by
the Artin map (see \cite{nx}, Chapter 2).

The following result resumes many useful formulas for computing the genus of a ray class field (see \cite{au00}, Section 1). 

\begin{teo}\label{ray}
Let $K$ be a function field over the constant field
$\mathbb{F}_q$ of genus $g_K$ and let $h_K$ be the divisor class number of $K$.
Let $S$ be a place of degree $d$ and $\mathfrak{m}=\sum_{i=1}^{k} m_i P_i$
be an effective divisor of $K$ where $P_i$ are distinct places of degree $n_i$ for $i=1,...,k$
such that $S\not\in Supp(\mathfrak{m})$ and $k\geq 1$ is a non negative integer.
Then the ray class field $K^{\mathfrak{m}}_S$ is a function field over
$\mathbb{F}_{q^d}$. The degree $[K^\mathfrak{m}_S:K]$ is equal to 
$$d h_K \prod_{i=1}^{k} \frac{(q^{n_i}-1)q^{(m_i-1)n_i}}{q-1}.$$
The genus $g_{K^\mathfrak{m}_S}$ of $K^\mathfrak{m}_S$ is given by
\begin{equation}\label{ray2}g_{K^\mathfrak{m}_S}=1+\frac{h_K \prod_i (q^{n_i}-1)}{2 (q-1)} (2 g_K-2+deg(\mathfrak{m})-\sum_i \frac{deg(P_i)q^{(m_i-1)n_i}}{q^{n_i}-1}).
\end{equation}
\end{teo}

\section{Ray class fields}


Let $h=h_K$ be the divisor class number of $K$. Then $h$ is the degree of every
maximal unramified abelian extension of $K$ with constant field $\mathbb{F}_q$. 
There are exactly $h$ such extensions 
of $K$ (see \cite{arta}, Chapter 8). We denote them by $K^0_1,\ldots,K^0_h$.

A similar result holds concerning also ramified extensions.
\begin{teo}\label{artintate}
Let $\mathfrak{m}=\sum_{i=1}^{t}m_iP_i$ be an effective divisor and let $n_i$ be the degree of $P_i$ 
for $i=1,\ldots,t$. We set $\mathfrak{m}=0$ if $t=0$. We set also 
$d=\frac{h_K}{q-1}\prod_{i=1}^{t}(q^{n_i}-1)q^{(m_i-1)n_i}$
if $t>0$ and $d=h_K$ otherwise. Then there are exactly $d$ abelian extensions of $K$ of degree $d$
with conductor $\mathfrak{m}$ and constant field $\mathbb{F}_q$.
\end{teo}

As before we denote such extensions by $K^{\mathfrak{m}}_1,\ldots,K^\mathfrak{m}_d$.
There is no conflict with the previous notation because the
result concerning unramified extensions can be seen as a special case of the previous Theorem.

\begin{proof}
In order to apply the Artin reciprocity Theorem we construct suitable subgroups of the Class group $C_K$.

Let $U_0$ be the subset of $J_K$ given by 
$$U_0=\{(x_P)_{P\in\mathcal{P}_K} \in J_K| x_P\in \hat{U}^*_P \mbox{ for all places } P\in\mathcal{P}_K\}$$ and let
 $U_\mathfrak{m}$ be the subset of $U_0$ given by
$$U_\mathfrak{m}=\{(x_P)_{P\in\mathcal{P}_K}\in U_0 | x_P\equiv 1 \mod t_i^{m_i} \mbox{ for all }i=1,\ldots,t\},$$
where $t_i$ is an uniformizer parameter at $P_i$. 
As before we set $U_\mathfrak{m}=U_0$ if $\mathfrak{m}=0$.
The field $K^*$ is canonically embedded in $J_K$ and we denote it again with $K^*$ as in the previous Section.
Let 
$C_\mathfrak{m}=U_\mathfrak{m}/(K^*\cap U_\mathfrak{m})$ be the classes of 
$U_\mathfrak{m}$ in $C_K$.

Let $D_0$ be the subgroup of $C_K$ of classes of ideles $x=(x_P)_{P\in\mathcal{P}_K}$ such that the divisor
$$Div(x)=\sum_{P\in\mathcal{P}_K}v_P(x_P)P$$
has degree $0$. 
It is easy to check that $D_0$ is well-defined 
 and $U_0\subseteq D_0$.

The following sequence is exact (see \cite{arta}, Chapter 8):
\begin{equation}\label{seqesatt}
0\rightarrow D_0\rightarrow C_K\rightarrow \mathbb{Z}\rightarrow 0,
\end{equation}
where the map $C_K\rightarrow\mathbb{Z}$ is the degree of the divisor and 
it is surjective by the Schmidt Theorem (see \cite{st93}, Chapter V).
Let $D$ be a divisor of degree~1. It is very easy to construct a class $x\in J_K$ such that $Div(x)=D$.
Let $[x]\in C_K$ be the class of $x$ in $C_K$. The subgroup generated by $C_\mathfrak{m}\cup [x]$ 
in $C_K$ has finite index $d$. 
 Let $a_1,\ldots,a_d$ be the representatives of the cosets of $C_\mathfrak{m}$ in $D_0$.
Then the subgroups $B_i$ of $C_K$ generated by 
$C_\mathfrak{m}\cup ([x]+a_i)$ are $d$ distinct subgroups of $C_K$ of index $d$ such that the image
onto $\mathbb{Z}$ in (\ref{seqesatt}) is surjective.

Let $K^{\mathfrak{m}}_1,\ldots,K^\mathfrak{m}_d$ be the function fields corresponding to
the subgroups $B_1,\ldots, B_d$ by the Artin map. 
 It is very easy to  show that $K^{\mathfrak{m}}_1,\ldots,K^\mathfrak{m}_d$ 
are all the abelian extensions of $K$ satisfying the hypothesis
of the Theorem. 
\end{proof}

\begin{oss}\label{sub1}
The proof of the previous Theorem shows that the extensions 
$K^\mathfrak{m}_1$, $K^\mathfrak{m}_2,\ldots,K_d^\mathfrak{m}$ of $K$
are all contained in the constant field extension of degree $d$ of any one of them, 
say $K_1^\mathfrak{m}\mathbb{F}_{q^{d}}$. In fact the compositum of function fields 
$K^\mathfrak{m}_i K^\mathfrak{m}_j$ corresponds to the intersections $B_{i,j}=B_i\cap B_j$ in $C_K$
by the Artin reciprocity map for $i,j\in\{1,\ldots,d\}$. The image of the valuation of $B_{i,j}$ by the degree map in 
(\ref{seqesatt}) is a subgroup of $\mathbb{Z}$ of finite index $d'|d$. In particular 
$K^\mathfrak{m}_i K^\mathfrak{m}_j= K^\mathfrak{m}_i\mathbb{F}_{q^{d'}}$.
\end{oss}

\begin{oss}\label{cyclicclass}
When the quotient group $D_0/C_\mathfrak{m}$ is cyclic we can say something more about the subextensions
of $K^\mathfrak{m}_i$ containing $K$ for $i=1,\ldots,d$. In fact let $l$ be a divisor of $d$. Then there is only
one subgroup $G$ of $D_0/C_\mathfrak{m}$ of index $l$. Let $g_1$, $\ldots$, $g_l$ be the cosets representatives
of $G$ in $D_0/C_\mathfrak{m}$. We denote by $F_i$ the fields 
corresponding by the Artin reciprocity map to the subgroups $G_i$ of $C_K$
generated by $G\cup ([x]+g_i)$ for $i=1,\ldots, l$. The field extensions $F_i/K$ are all the abelian  extensions of degree
$l$ unramified outside $\mathfrak{m}$ with constant field $\mathbb{F}_q$ for $i=1,\ldots,l$.
\end{oss}

\begin{cor}\label{precisi}Let $\mathfrak{m}$ be 
an effective divisor and $d$ a positive integer as in the previous Theorem.
Let $P$ be a place of $K$ and denote its degree by $d'$. Let $l$ be the positive integer $\gcd(d,d')$ 
and $P_i|P$ be a place of $K^\mathfrak{m}_i$ over $P$ for $i\in\{1,\ldots,d\}$.
If $D_0/C_\mathfrak{m}$ is a cyclic group then $f(P_i|P)=1$ in at most $l$
 such extensions $K^\mathfrak{m}_i/K$.
\end{cor}
\begin{proof}
Assume that the place $P$ is totally split in $K^\mathfrak{m}_i/K$ for at least
one $i\leq d$, otherwise the proof would be trivial. Then $P$ is split in
$K^\mathfrak{m}_j/K$ for $j\not = i$ if and only if $P$ is totally split in
the compositum $K^\mathfrak{m}_iK^\mathfrak{m}_j/K$. But
$K^\mathfrak{m}_iK^\mathfrak{m}_j=K^\mathfrak{m}_i\mathbb{F}_{q^a}$
for a suitable integer $a|d$ by  Remark \ref{sub1}. By the properties of the
constant field extensions this is possible only when
$a|d'$ and so $a|l$ and $K^\mathfrak{m}_j\subseteq K^\mathfrak{m}_i\mathbb{F}_{q^l}$.

It follows from the proof of the previous Theorem that 
$$l\cdot ([x]+a_i)\subseteq B_j$$ and so
$l\cdot (a_i-a_j)\in C_\mathfrak{m}$ and the class of $l\cdot a_j$ in the quotient group 
$D_0/C_\mathfrak{m}$ is the class of $l\cdot a_i$. 
It is very easy to see that when $D_0/C_\mathfrak{m}$ is a cyclic group 
 there are at most $l$ such classes $a_j\in D_0/C_\mathfrak{m}$ and so there are at most $l$ corresponding 
fields extensions by the previous Theorem.
\end{proof}

The previous Corollary can be generalized in the following result.

\begin{cor}\label{precisi2}
Assume the quotient group $D_0/C_\mathfrak{m}$ be a cyclic group of order $d$ as in Corollary \ref{precisi}. Let $s$
be a prime dividing $d$ and let $t$ be the maximal power of $s$ dividing $d$. Let $F_i/K$ be the 
extensions of degree $t$ for $i=1,\ldots,t$ as in Remark 
\ref{cyclicclass}. Let $P$ be a place of $K$ of degree $d'$ and $P_i|P$ be a place
of $F_i$ over $P$. Let $l$ be the $\gcd(d',t)$ and let $c\geq 0$ be the exponent such that $\frac{t}{l}=s^c$. 
Assume $c\geq 1$.
Then for all integers $j=1,2,\ldots, c$, the integer $s^j$ divides $f(P_i|P)$ 
 in at least 
 $l(s^c-s^{j-1})$ such extensions $F_i/K$.
\end{cor}

\begin{proof}
Let $j'$ denote the number $ls^{c-j+1}$ and $E_1/K$, $\ldots$, $E_{j'}/K$ be the extensions of $K$ unramified
outside $\mathfrak{m}$ of degree $j'$ over $K$ by Corollary \ref{precisi}.
If $s^j\not |f(P_i|P)$ for a certain $i\in\{1,\ldots,t\}$ then the Frobenius $Frob(P)$ of $P$ in $F_i/K$ has 
order dividing $s^{j-1}$. Let $E_{i'}/K$ be the only subfield of $F_i$ of degree $j'$ over $K$ and let 
$P'_{i'}$ be the place under $P_i$ in $E_{i'}$. Then $Frob(P'_{i'})=Frob(P_i)^{j-1}=1$ so
 $f(P'_{i'}|P)=1$. By Corollary \ref{precisi} there are at most $l$ such extensions $E_i/K$ such that $f(P_i'|P)=1$,
 say, $E_1/K$, $\ldots$, $E_l/K$. There are exactly $s^{j-1}$ extensions $F_i/K$
over each $E_{i'}$ so $s^j\not |f(P_i|P)$ in at most $ls^{j-1}$ extensions $F_i/K$ and the Corollary follows.
\end{proof}

\begin{oss}\label{precisi3}
In the previous Corollary when $j=c$ we obtain that $\frac{t}{l}$ does not divide $f(P_i|P)$ in atmost $\frac{t}{s}$
extensions $F_i/K$.
\end{oss}

\section{A refinement of the Clark-Elkies bound}

When $K$ is the rational function field $\mathbb{F}_q(x)$, we get as a 
Corollary a result similar to the one of Clark and Elkies cited in \cite{hola05} (see \cite{st} Section 4.1 for more details)
but we can improve that result for large $n$ by considering
ray class field extensions of $K$ with conductor given by a sum of different places.

In the sequel we denote by $K$ the rational function field over 
$\mathbb{F}_q$. 
The number of places of degree $t$ of $K$ is denoted by $a_t$ for any integer $t>0$.
It is very easy to check by induction that, for all $n\geq 1$,
\begin{equation}\label{bond}
\sum_{d<n} a_d\leq q\cdot \frac{q^n}{n}.
\end{equation}

The next Lemma shows that there are many function fields without places of small degree when
we consider ray class field extensions of $K$. 

\begin{lemmino}\label{enumeq}
Let $C_1>0$ and $C_2>0$ be two real positive constants (not depending on $n$) with $C_2<1$.
Let $m>\log_q(n)$ be a prime number 
and let $\alpha\leq a_m$ be a positive integer.
Let $\mathfrak{q}_1,\ldots,\mathfrak{q}_\alpha$ be distinct places of $K$ of degree $m$ and 
let $\mathfrak{m}$ be the divisor $\sum_{i=1}^\alpha \mathfrak{q}_i$.
We set $d=\frac{(q^m-1)^\alpha}{q-1}$. Let $K^\mathfrak{m}_1,\ldots,K^\mathfrak{m}_d$ 
be the abelian extensions of degree $d$
unramified outside $\mathfrak{m}$ as in Theorem \ref{artintate}.
Then there is a constant $n_0$ such that when $n>n_0$ and 
$\alpha>C_1\frac{n}{\log_q(n)}$ 
then there are at least $C_2 d$ function
field extensions $K^\mathfrak{m}_i$ of $K$ such that
the inertia index $f(P_i|P)$ is greater than $\frac{n}{deg(P)}$ whenever $P$ is a place
of $K$ of degree $deg(P)<\frac{n}{\log_q(n)}$ and $P_i$ is a place of $K^\mathfrak{m}_i$ over $P$.
\end{lemmino}

In the proof we use a well-known Lemma and an easy consequence.
\begin{lemmino}\label{eul}
Let $s$ and $m$ be odd prime numbers and let $q$ be a prime power
such that $s|\,\frac{q^m-1}{q-1}$ but $s\not | q-1$. Then 
$s=2 a m+1$ for a suitable integer $a>0$. In particular $s>2 m$.
\end{lemmino}

\begin{cor}\label{atmost}
There is a constant $c_q>0$ such that when $m>c_q$ is a prime then there are
 at most $m$ distinct primes dividing $\frac{q^m-1}{q-1}$
and these primes are all greater than $2m$.
\end{cor}

\begin{proof}[Proof of Lemma \ref{enumeq}]
Let $i$ be an element in $\{1,\ldots,d\}$ such that $K^\mathfrak{m}_i/K$ is a function field extension
with $f(P_i|P)<\frac{d}{deg(P)}$ for at least one place $P$ of $K$ of degree
smaller than $\frac{n}{\log_q(n)}$. 
 We estimate the number of such extensions.

Let $k$ be the integer $\frac{q^m-1}{q-1}$. Let $j$ be an integer in $\{1,\ldots,\alpha\}$ and
let $t$ be a power of the prime number $s$ such that $t$ divides $k$.
Consider the subextensions of $K^{\mathfrak{q}_j}_i\subseteq K^\mathfrak{m}_i$ 
totally ramified in $\mathfrak{q}_j$ of degree $t$ for $j\in \{1,\ldots,\alpha\}$. 
 Let $d'$ denote the degree of $P$ and let
$P_{i,j}$ be the place of $K^{\mathfrak{q}_j}_i$ under $P_i$ with $P_{i,j}|P$.
 Let $l$ be the integer $\gcd(t,d')$. It is easy to see, multiplying all the maximal prime power $t$ dividing $k$, that if 
for every prime power divisor $t$ of $k$ the number
$\frac{t}{l}$ divides $f(P_{i,j}|P)$ for at least one $j\leq \alpha$
then $$k | f(P_i|P)\gcd(k,d')$$
and so $$f(P_i|P)\geq \frac{n}{d'},$$ because $k>n$ and $d'\geq \gcd(k,d')$.
It follows that if $f(P_i|P)<\frac{d}{deg(P)}$ then there is at least one prime power $t$ dividing $k$ such that
$\frac{t}{l}\not | f(P_{i,j}|P)$ for all $j=\{1,\ldots,\alpha\}$.
For this reason, given a prime power $t$ dividing $k$, it will be enough to estimate only the number of extensions
$K^\mathfrak{m}_i/K$ such that $\frac{t}{l}\not | f(P_{i,j}|P)$ for all $j=1,\ldots,\alpha$.

The extensions $K^{\mathfrak{q}_j}_i/K$ are
cyclic for $j\in \{1,\ldots,\alpha\}$. By Remark \ref{precisi3}
there are at most $\frac{t}{s}$ distinct extensions $K^{\mathfrak{q}_j}_i/K$ of degree $t$
totally ramified in $\mathfrak{q}_j$ such that $\frac{t}{l}\not | f(P_{i,j}|P)$.
 It follows that there are at most $\big( \frac{k}{s}\big)^\alpha$ different extensions
$K^{\mathfrak{q}_1}_i \cdots K^{\mathfrak{q}_\alpha}_i$ of $K$
such that $\frac{t}{l}\not | f(P_i|P)$ when $P$ is unramified. So we see that there are at most
$$\frac{d}{s^\alpha}$$
extensions $K^{\mathfrak{m}}_i/K$ with $deg(P_i)<n$. 

Now we consider the case $P=\mathfrak{q}_h$, for a certain $h\in\{1,\ldots,\alpha\}$, is a ramified place. 
We consider $\mathfrak{m}'=\mathfrak{m}-P$. For a similar reasoning as above we get at most
$$\frac{(q^m-1)^{\alpha-1}}{(q-1)s^{\alpha-1}}$$ extensions $K^{\mathfrak{m}'}_j$ for 
$j\in\{1,\ldots, \frac{(q^m-1)^{\alpha-1}}{q-1}\}$ such that $f(P'_j|P)<\frac{n}{deg(P)}$, where $P'_j$
is a place of $K^{\mathfrak{m}'}_j$ over $P$. But $K^{\mathfrak{m}'}_j\subseteq K^{\mathfrak{m}}_i$
for $q^m-1$ suitable $i\in\{1,\ldots,d\}$ and $f(P'_j|P)\leq f(P_i|P)$ so there are at most 
$$\frac{d}{s^{\alpha-1}}$$ extensions $K^{\mathfrak{m}}_i/K$ 
of $K$ with $f(P_i|P)<\frac{n}{deg(P)}$ when $P\in Supp(\mathfrak{m})$ is ramified.

Now we sum the number of all such extensions for all the places $P$ of $K$, 
ramified or not, of degree smaller than $\frac{n}{\log_q(n)}$
 and for all prime $s|k$. 
 So we prove the following inequality:
\begin{equation}\label{alpha}
\sum_{s|k}\sum_{i=1}^\alpha \frac{d}{s^{\alpha-1}}+\sum_{deg(P)<\frac{n}{ \log_q(n)}}\sum_{s|k}
\frac{d}{ s^{\alpha}}<(1-C_2) d,
\end{equation}
where $P$ runs over the unramified places of $K$ of degree smaller than $\frac{n}{\log_q(n)}$.
The left hand side in (\ref{alpha}) is bounded by
$$m\alpha\frac{d}{(2m)^{\alpha-1}}+mq\cdot q^\frac{n}{\log_q(n)}\frac{d}{(2m)^\alpha}$$ 
by (\ref{bond}), Lemma \ref{eul} and Corollary \ref{atmost}. So we prove that
$$(2m)^{\alpha}> \frac{qm}{1-C_2}(2m\alpha+ q^{\frac{n}{\log_q(n)}}),$$
or also 
\begin{equation}\label{alpha2}\alpha\log_q(2m)>\log_q(q^\frac{n}{\log_q(n)}+2m\alpha)+\log_q(\frac{m}{1-C_2})+1.\end{equation}
It is very easy to check the last inequality in fact the right hand side is smaller than
$$\frac{n}{\log_q(n)}+\log_q(2m\alpha)+\log_q(\frac{m}{1-C_2})+1,$$
because the logarithm is a convex function and so
 (\ref{alpha2}) holds when $n$ is large because $\alpha >C_1\frac{n}{\log_q(n)}$ by hypothesis.
\end{proof}

The proof of the following Lemma follows directly by the Hurwitz genus formula. 
 It is a generalization of (\ref{ray2}).

\begin{lemmino}\label{genq}
Let  $\mathfrak{q}_1, \ldots,\mathfrak{q}_h$ be distinct places
of $K$ of degree $t_1,\ldots,t_h$ respectively.
Let $p_1,\ldots,p_h$ be positive integers such that $p_i|\frac{q^{t_i}-1}{q-1}$ for $i=1,\ldots,h$.
Let $F_i/K$  be ray class field extensions of degree
$p_i$ totally ramified in $\mathfrak{q}_i$ for $i=1,\ldots,h$. Then the genus $g_L$ of the compositum field 
$L=F_1\cdots F_h$ is smaller than 
$$g_L\leq\frac{1}{2}\sum_{i=1}^{h}t_i\prod_{j=1}^{h} p_j.$$
\end{lemmino}

\begin{prop}\label{enumeq2}
Let $m$ and $l$ be distinct prime numbers with $l$ and $m$ greater than $3\log_q(n)$
and let $\alpha$ and $\beta$ be positive integers with $\alpha\leq a_m$ 
 and $\beta\leq a_l$. 
 Let $C_1>0$ be a real constant 
and let $C_2>0$ be a real constant with $C_2<1$ as in Proposition \ref{enumeq}.
Let $\mathfrak{q}_1,\ldots,\mathfrak{q}_\alpha$ (resp. $\mathfrak{p}_1,\ldots,\mathfrak{p}_\beta$)
be distinct places of $K$ of degree $m$ (resp. $l$) with $\alpha>C_1\frac{n}{\log_q(n)}$.
Let $\mathfrak{m}$ be the effective divisor $\sum_{i=1}^\alpha\mathfrak{q}_i+\sum_{j=1}^\beta\mathfrak{p}_j$.
Let $k_1$ and $k_2$ be the integers $\frac{q^m-1}{q-1}$ and $\frac{q^l-1}{q-1}$ respectively
and set $d=\frac{(q^m-1)^\alpha(q^l-1)^\beta}{q-1}$. Assume that $k_1$ and $k_2$ are both prime to $q-1$. 

Then there is an integer $n_0$ such that when $n>n_0$ and 
$$\frac{C_2}{2}d>\frac{q\cdot q^n}{n},$$
there is a function field extension $K^\mathfrak{m}_i/K$ for a certain 
$i\in\{1,\ldots,d\}$ without places of degree smaller than $n$.
\end{prop}

\begin{proof}
We may assume that $l$ and $m$ are smaller than $\frac{n}{\log_q(n)}$ otherwise the proof would
be  more easy. By Lemma \ref{enumeq} there are at least $C_2d$ function field extensions 
$K^\mathfrak{m}_i/K$ for $i=1,\ldots,d$ such that $deg(P)f(P_i|P)\geq n$ whenever
$deg(P)<\frac{n}{\log_q(n)}$ and $P_i$ is a place over $P$.
 In one of these field extensions $K^\mathfrak{m}_i$ of $K$ there is a 
place of degree smaller than $n$ only if there is a place $P$ of $K$
of degree $d'<n$ with $d'\geq \frac{n}{\log_q(n)}$ such that $P$ is
totally split in $K^{\mathfrak{q}_j}_i/K$ for all $j\in\{1,\ldots,\alpha\}$ 
and in $K^{\mathfrak{p}_h}_{i}/K$ for all $h\in\{1,\ldots,\beta\}$ by Lemma \ref{eul},
where $K^{\mathfrak{q}_j}_i$ and $K^\mathfrak{p}_h$ are the ray class fields of $K$ with conductor
$\mathfrak{q}_j$ and $\mathfrak{p}_h$, respectively, contained in $K^\mathfrak{m}_i$. 
We are going to estimate the number of such function field extensions $K^\mathfrak{m}_i/K$.
 
For a fixed $j\leq \alpha$ we consider $K^{\mathfrak{q}_j}_i/K$ for $i\in\{1,\ldots,k_1\}$.
There are at most $d_1=\gcd(d',k_1)$ function field
extensions $K^{\mathfrak{q}_j}_i/K$ such that $P$ is
totally split by Corollary \ref{precisi}. Similarly for a fixed $h\leq \beta$
there are at most $d_2=\gcd(d',k_2)$ function
field extensions $K^{\mathfrak{q}_h}_i/K$ with $i\in\{1,\ldots,k_2\}$
such that $P$ is totally split. We denote by $d''$ the greatest common divisor $\gcd(q-1,d')$. It follows that there
are at most $d_1^\alpha d_2^\beta d''^{\alpha+\beta-1}$ extensions $K^\mathfrak{m}_i/K$ 
with $i\in\{1,\ldots,d\}$ such that $P$ is totally split. 

Let $A_{d_1,d_2,d'}$  be the number of places of $K$ of degree $d'$ totally split in 
all the subextensions of degree $d_1 d''$ (resp. $d_2 d''$) 
of  the ray class fields $K^{\mathfrak{q}_j}_i$ for $i\in\{1,\ldots,k_1\}$ and $j\in\{1,\ldots,\alpha\}$
(resp. $K^{\mathfrak{p}_h}_i$ for $i\in\{1,\ldots,k_2\}$ and $h\in\{1,\ldots,\beta\}$). 
Then $$A_{d_1,d_2,d'}\leq \frac{q^{d'}}{d'd_1^{\alpha}d_2^\beta d''^{\alpha+\beta-1}}+2\frac{g_L}{d_1^\alpha d_2^\beta d''^{\alpha+\beta-1}}q^{d'/2}+deg(\mathfrak{m})$$
by 
the Chebotarev Theorem (see \cite{ms94}), where $L$ is the compositum of the subextensions of degree $d_1$ and
$d_2$ of $K^{\mathfrak{q}_j}_i$ and $K^{\mathfrak{p}_h}_i$. By Lemma \ref{genq} we get
$$A_{d_1,d_2,d'}\leq \frac{q^{d'}}{d'd_1^{\alpha}d_2^\beta d''^{\alpha+\beta-1}}+(q^{d'/2}+1)(m\alpha+l\beta).$$

 By the previous Proposition there are at least $C_2 d$  
distinct extensions $K^\mathfrak{m}_i/K$ 
such that $f(Q|P)deg(P)>n$ when $deg(P)<\frac{n}{\log_q(n)}$ but there are at most
$$\sum_{d'=\frac{n}{\log_q(n)}}^{n-1} A_{d_1,d_2,d'}d_1^\alpha d_2^\beta d''^{\alpha+\beta-1}$$
 extensions $K^\mathfrak{m}_i/K$ with at least one totally split place of degree 
$ d'$ by Lemma \ref{precisi}.
In particular this number is smaller than
$$\sum_{d'=\frac{n}{\log_q(n)}}^{n-1}\frac{q^{d'}}{d'}+(q^{d'/2}+1)(\alpha m+\beta l)d_1^\alpha d_2^\beta d''^{\alpha+\beta-1}.$$
 But $$d_1d''\leq d'<n<k_1^{1/3}$$ and similarly for $d_2d''$. 
Moreover $\alpha m+\beta l<2\log_q(d)$.
 It follows that there are at most
\begin{equation}\label{final}
 q\frac{q^n}{n}+2nq^{n/2}\log_q(d)d^{1/3}
\end{equation}
extensions $K^\mathfrak{m}_i/K$ such that at least one point of degree $d'<n$ is totally split.
The right hand side in (\ref{final}) is smaller than $C_2 d$ if
$$q\frac{q^n}{n}<\frac{C_2}{2} d$$
and $$2nq^{n/2}\log_q(d)d^{1/3}<\frac{C_2}{2} d.$$
The first condition holds by hypothesis, the second one holds when $n$ is large because $d>q^\frac{n}{\log_q(n)}$. 
 So there is at least one function field extension $K^\mathfrak{m}_i/K$ without places of degree smaller than $n$.
\end{proof}

In order to prove Theorem \ref{gol} we choose $l$ and $m$ greater than $3\log_q(n)$ 
but smaller than $C\log_q(n)$ for a suitable constant $C>0$ and we find suitable
 $\alpha$ and $\beta$ smaller than $n$ with $\alpha$ or $\beta$ greater than $C_1\frac{n}{\log_q(n)}$
 for an other suitable $C_1>0$ such that the integer $$d=\frac{(q^m-1)^\alpha (q^l-1)^\beta}{q-1}$$
 is bigger than $4q\frac{q^n}{n}$ but smaller than
$C'4q\cdot\frac{q^n}{n}$ for a suitable constant $C'>1$ (not depending on
$n$). 
In fact, when $d\leq C'4q\cdot\frac{q^n}{n}$ then $m\alpha+l\beta\leq n$ when $n$ is large and
the genus $g$ of $K^\mathfrak{m}_i$ is bounded by $g\leq \frac{m\alpha+l\beta}{2} d$ 
for all $i\in\{1,\ldots,d\}$, by (\ref{ray2}), so $g\leq \frac{n}{2}d\leq 2C'q\cdot q^n$.
We will see that $C'=q$ is a possible choice for $C'$. 

The existence of suitable $\alpha$ and $\beta$ is proved by the next Lemma. 
\begin{lemmino}
Let $l$ and $m$ be coprime numbers with $l<m<2l$. Then there is a constant $l_0$,
such that when $l>l_0$ then for any real number $r$ greater than $q^{2m^3}$
there are two positive integers $\alpha$ and $\beta$ such that
\begin{equation}\label{richiesta}
r<\frac{(q^m-1)^\alpha(q^l-1)^\beta}{q-1}< rq.
\end{equation}
\end{lemmino}

\begin{proof}
Let $R$ be the real number $\log_q(rq)+\log_q(q-1)$.
Taking logarithm of both sides in (\ref{richiesta}) we get the equivalent condition
$$R-1<\alpha q_m+\beta q_l\leq R,$$
where $q_m$ and $q_l$ denote the real numbers $\log_q(q^m-1)$ and $\log_q(q^l-1)$.

By means of the Farey series (see \cite{hawr}, Chapter III) we can find positive integers 
$h$ and $k$ with $0<h<k<m$ such that the real number 
$$v=kq_l-hq_m$$ 
satisfies $\frac{1}{2}<v<1$. In fact $\frac{h}{k}$ is the rational number
preceding $\frac{l}{m}$ in the Farey series and $\frac{h}{k}<\frac{q_l}{q_m}<\frac{l}{m}$ when $l$ is large 
compared to $q$. In particular $v<kl-hm=1$ by an elementary properties of the Farey series and $v>\frac{1}{2}$ otherwise
$$\frac{q_l}{q_m}-\frac{h}{k}=\frac{v}{kq_m}<\frac{1}{2kq_m}$$
 so 
$$\frac{l}{m}-\frac{q_l}{q_m}+\frac{1}{2kq_m}>\frac{l}{m}-\frac{h}{k}=\frac{1}{km}$$
and so $$\frac{l}{m}-\frac{q_l}{q_m}>\frac{1}{km}-\frac{1}{2kq_m}>\frac{1}{4m(m-1)}$$
and we get a contradiction because $\frac{l}{m}-\frac{q_l}{q_m}<\frac{1}{4m(m-1)}$ when $l$ is large.

Let $c$ be the integer $[\frac{R}{q_m}]$ and let $z$ be the real number $cq_m$. If $z>R-1$ then 
we choose $\alpha=c$ and $\beta = 0$  and the Lemma follows. Otherwise we define the succession
$z_i=z+iv$ for all integers $i\geq 0$. Let $j$ be the minimum integer such that $z_j>R-1$.  
Then $z_j<R$ because $v<1$ and so $j<\frac{c}{h}$ otherwise $z_j$ would be greater than $R$,
because $v>\frac{1}{2}$ and $R>2m^3$, but this is not the case. We choose
$\alpha=c-jh$ and $\beta=jk$ and the Lemma follows.
\end{proof}


\begin{proof}[Proof of Theorem \ref{gol}] 
We assume before $q=p$ is a prime.

Choose prime numbers $l$ and $m$ and two positive integers $\alpha$ and $\beta$
satisfying (\ref{richiesta}) in the previous Lemma with $r=4p\frac{p^n}{n}$. Such choice of $r$
verifies the hypothesis of the Lemma when $n$ is large and $l$ and $m$
are smaller than $C\log_p(n)$ for a constant $C>0$.
By the Bertrand postulate there are at least two primes smaller than $C\log_p(n)$ when $C\geq 12$
so there are such integers.

It is easy to see that $\alpha<a_m$ and $\beta<a_l$ if $l$ and $m$ are greater than $3\log_p(n)$ 
otherwise $p^{m\alpha+l\beta}$ would be greater than $p^{n^3}$ and it would not satisfy (\ref{richiesta}).
In a similar way we see that $\alpha$ or $\beta$ is greater than, say, $\frac{1}{48}\frac{n}{\log_p(n)}$ otherwise
$p^{m\alpha+l\beta}$ would be smaller than $p^{n/2}$ in contrast with (\ref{richiesta}). So we
can apply Proposition \ref{enumeq2} with $C_1=\frac{1}{48}$. We get a  function field without places of degree
smaller than $n$ for all $n>n_0$ for a suitable constant $n_0$. 
 We have already seen that the genus of such function field is smaller than $\frac{1}{2(p-1)}p^n$ by (\ref{ray2}). 
Let $C_p$ be the constant $\frac{1}{2(p-1)}p^{n_0}$. Then there is a function field with constant field $\mathbb{F}_p$
without places of degree smaller than $n$ of genus smaller than $C_pp^n$ for all integer $n>0$.

Now let $q=p^c$ be a prime power of $p$. By the previous case there is a function field $K$ of genus 
$g_K\leq C_p p^{cn}=C_p q^n$ over $\mathbb{F}_p$ without places of degree smaller of $cn$. The constant field
extension $K\mathbb{F}_{q}$  is a function field over $\mathbb{F}_q$ with the same genus
without places of degree smaller than $n$. This concludes the proof.
\end{proof}

\section{Tables}

We list examples of curves over $\mathbb{F}_q$ without points of degree
$d'$ such that $d'\leq n$ when $q=2$ and $n< 20$.

The integer $d$ in the table is the degree of a function field extension $K/\mathbb{F}_q(x)$
of the rational function field with genus $g$ and constant field $\mathbb{F}_q$.
In this table the field $K$ is always a subfield of the ray class field $K^\mathfrak{m}_S$
of conductor $\mathfrak{m}$. The irreducible polynomials in the forth column correspond to the 
places in the support of $\mathfrak{m}$ with multiplicity.
The polynomial in $\mathbb{F}_q(x)$ corresponding to the place $S$ totally split 
in $K^\mathfrak{m}_S/\mathbb{F}_q(x)$ is showed in the last column.
\large{
\begin{landscape}
\begin{center}
\bfseries{Pointless curves for $\mathbf{q=2}$}
\end{center}

\begin{center}

\begin{tabular}{|l|c|c|c|c|}

\hline
$\mathbf{n}$ & $\mathbf{g}$ & $\mathbf{d}$& $\mathfrak{m}$ & $\mathbf{S}$\\
\hline
\hline
1 & 2 & 2& $(x^3+x+1)^2$ & $(x^3+x^2+1)$\\
\hline
2 & 3 & 7& $(x^3+x+1)$& $(x^4+x+1)$\\
\hline
3 & 4 & 5& $(x^4+x+1)$& $(x^7+x^4+1)$\\
\hline
5& 12 & 7& $(x^6+x^4+x^3+x+1)$& $(x^8+x^5+x^3+x^2+1)$\\
\hline
7& 48 & 17& $(x^8+x^7+x^6+x+1)$& $(x^9 + x^7 + x^5 + x^2 + 1)$\\
\hline
8&78&$7\cdot 7$ &$(x^3+x^2+1,x^3+x+1)$&$(x^9 + x^7  + x^2 + x + 1)$\\
\hline
9& 120&$31$&$(x^{10}+x^3+1)$&$(x^{11} + x^9 + x^7 + x^2 + 1)$\\
\hline
11& 362 & $15\cdot 7$& $(x^4+x+1,x^6+x^5+x^3+x^2+1)$ & $(x^{13}+ x^8 + x^5 + x^3 + 1)$\\
\hline
12& 588 & $31\cdot 7$& $(x^5+x^2+1,x^3+x+1)$& $(x^{13} + x^{12} + x^{10} + x^7 + x^4 + x+ 1)$\\
\hline
13&1480&$31\cdot 15 $&$(x^5+x^2+1,x^4+x+1)$&$(x^{14} + x^{13} + x^5 + x^4 + x^3 + x^2 + 1)$\\
\hline
14&3342&$127\cdot 7$&$(x^7+x+1,x^3+x+1)$&$(x^{15} + x^{14} +x^{13} +x^7 +x^6 +x^4 +x^2 + x+ 1)$\\
\hline
15& 8940&$73\cdot 17$&$(x^{9}+x^4+1,x^8+x^5+x^3+x^2+1)$&$(x^{16}+x^{14}+x^{13}+x^{11}+x^{10}+x^7+x^4+x+1)$\\
\hline
16&19861&$23\cdot 89$&$(x^{11}+x^6+x^5+x^2+1,x^{11}+x^9+1)$&$(x^{18} + x^{17} + x^{11} + x^9 + x^7 + x^4 + 1)$\\
\hline
17&41440&$89\cdot 63$&$(x^{11}+x^9+1,x^6+x+1)$&$(x^{18} + x^{17} + x^{16} + x^{11} + x^9 + x^4 + 1)$\\ 
\hline
18&89415&$127\cdot 89$&$(x^7+x+1,x^{11}+x^9+1)$&$(x^{19} + x^{18} + x^{15} + x^{14}+x^{11}+x^7+x^3+ x + 1)$\\
\hline
19&95886 & $127\cdot 127$&$(x^7+x+1,x^7+x^6+1)$&$(x^{20} + x^{19} + x^{15} + x^{14} + x^{13} + x^2 + 1)$\\
\hline
\end{tabular}
\end{center}
\end{landscape}
}

\subsection*{Acknowledgements}
This research is part of my Phd thesis at the Universit\`a di Roma 'Sapienza'.
I am grateful to Prof. Ren\'e Schoof, Universit\`a di Roma 'Tor Vergata', for his patience and support.
Part of this work was done when I was a visiting student at the Mathematical Institute at Leiden University.

\end{document}